\documentclass[11pt]{article}
\usepackage{t1enc}
\usepackage[latin1]{inputenc}
\usepackage{graphicx}
\usepackage{amsthm}
\usepackage[T1]{fontenc}
\usepackage{amsmath}
\usepackage{amssymb}
\usepackage{stmaryrd}
\usepackage{makeidx}
\usepackage{scrtime}
\usepackage{fancyhdr}
\usepackage{enumerate}
 \usepackage[all]{xypic}
\usepackage[mathscr]{eucal}
\usepackage[german,UKenglish]{babel}
\usepackage{longtable}
\usepackage[small,bf]{caption}

\DeclareMathOperator{\spann}{span}\DeclareMathOperator{\id}{id}

   \DeclareMathOperator{\Spin}{Spin}

\DeclareMathOperator{\Sp}{Sp}

\DeclareMathOperator{\dist}{dist}

 \theoremstyle{plain}
\newtheorem{satz}{Theorem}[section]
\newtheorem{lemma}[satz]{Lemma}

\newtheorem{korr}[satz]{Corollary}

\theoremstyle{definition}
\newtheorem{defi}[satz]{Definition}
 \theoremstyle{remark}

\numberwithin{equation}{section}
\newcommand{\C}{\mathbb{C}}

\def\bigset(#1|#2){\big\{#1\, \big|\, #2\big\}}
\def\biggset(#1|#2){\bigg\{#1\, \bigg|\, #2\bigg\}}
\def\Bigset(#1|#2){\Big\{#1\, \Big|\, #2\Big\}}
\def\vectdreik(#1;#2;#3){\text{\tiny $\begin{pmatrix} #1 \\ #2 \\ #3 \end{pmatrix}$}}
\def\vectdrei(#1;#2;#3){\begin{pmatrix} #1 \\ #2 \\ #3 \end{pmatrix}}

\newcommand{\dirac}{\mathcal{D}}

\newcommand{\ci}{\mathtt{i}}

\def\ddxi,#1{\frac{\partial}{\partial x_{#1}}}

\newcommand{\HQ}{\mathbb{H}}
\newcommand{\R}{{\mathbb{R}}}
\newcommand{\Z}{{\mathbb{Z}}}

\newcommand{\eps}{\varepsilon}
\newcommand{\forget}[1]{{\huge$\mathbf{\cdots}$}}

\newcommand{\Spe}{\Sp\!{\scriptstyle(1)}}

\newcommand{\ddt}{\frac{\partial}{\partial t}\Big|_0}

\def\vekthq(#1,#2,#3,#4){\left(\begin{array}{c} #1 \\ #2 \\ #3 \\ #4 \end{array}\right)}
\newcommand{\spe}{\text{\textbf{sp}}{\scriptstyle(1)}}

\newcommand{\rrspace}{\R/_{2\pi \Z}\,\times\, \R/_{2\pi \Z}\, \times\, ]0,1[}
\newcommand{\Hsl}{H^{\text{sl}}}
\newcommand{\Xsl}{X^{\text{sl}}}
\newcommand{\Ysl}{Y^{\text{sl}}}
\newcommand{\slc}{\text{\textbf{sl}}\big(2,\C\big)}
\newcommand{\CSpe}{\mathcal{C}^\infty\big(\Spe;\, \HQ\big)}
\newcommand{\matrixtt}[4]{\left(\begin{array}{cc} #1 & #2 \\ #3 & #4 \end{array}\right)}

\newcommand{\ev}[1]{{\mathbf{e_{#1}}}}
\newcommand{\met}[2]{\left\langle #1,\, #2 \right\rangle}
\newcommand{\ket}[1]{\big|#1 \big\rangle}

\title{Eigenspaces of the Spin Dirac operator over $S^3$}
\author{J. Fabian Meier\\ Mathematisches Institut, Universität Bonn \\ Endenicher Allee 60, 53115 Bonn \\ Kontakt:  brief@FabianMeier.de}

\begin{document}
\maketitle
\begin{abstract}
We calculate the spectrum and a basis of eigenvectors for the Spin Dirac operator over the standard 3-sphere. For the spectrum, we use the method of Hitchin which we transfer to quaternions and explain in more detail. The eigenbasis (in terms of polynomials) will be computed by means of representations of $\slc$.

\end{abstract}

The spectrum of the Spin Dirac operator over the round $n$-sphere was investigated in several papers before. The first calculation for arbitrary $n$ was made by Sonja Sulanke in her PhD thesis. Christian Bär followed a simpler approach in \cite{baerspace}, which is also capable of producing the eigenspaces out of killing spinors and harmonic polynomials.

Here we use an approach more in Sulanke's spirit which concentrates on the representation theory of $\Spe$.

We consider $S^3= \Spe$\index{Spe@$\Spe$} as canonically embedded into $\HQ$. The tangent space $T_{e_0} \Spe$ at the neutral element $e_0$ will be identified as usual with the Lie algebra $\spe$\index{spe@$\spe$}. As a vector space, $\spe$ can be seen as the imaginary quaternions in $\HQ$.

Our first goal is to calculate the spectrum of the Dirac operator $\dirac$. We will follow the path which has been paved by \cite{hitchin}, but translate it into the language of quaternions.

Hitchin used a metric which depends on three parameters $\lambda_1, \lambda_2, \lambda_3$. We will only use the standard metric ($\lambda_1=\lambda_2=\lambda_3 = 1$); everything can be calculated in the same way using the more general metric but we want to avoid an unnecessary amount of notation.

Unlike Hitchin, we are also interested in the calculation of a concrete basis of eigensections. For that purpose we have to transform the results of Hitchin, given in an abstract representation space, to the ``world of sections''. The required isomorphism will be constructed using the complexified Lie algebra of $\Spe$ and some elementary combinatorics.

\section{Definitions}
\label{sec:s3def}

$\Spe$ will be equipped with the canonical metric, coming from the embedding into $\HQ\cong \R^4$. The Lie algebra $\spe$ is spanned by $\{e_1, e_2, e_3\}$. They can be considered as left-invariant vector fields, which will be called $\ev1, \ev2, \ev3$.

The Levi-Civita connection can be computed by the standard formula (see e.g. \cite{doCarm}, p.55):
\begin{align*}
    \met{X}{\nabla_ZY} &=  Z\cdot \met{X}{Y} + \met{Z}{[X,Y]} \\
& + Y\cdot \met{X}{Z} +  \met{Y}{[X,Z]} \\
& - X\cdot \met{Y}{Z}  - \met{X}{[Y,Z]}.
\end{align*}
For left-invariants vector fields $X,Y,Z$ and a left-invariant metric, the terms of the first column vanish, so we get:
\begin{align*}
  2\met{X}{\nabla_Z Y} &= \met{Z}{[X,Y]} + \met{Y}{[X,Z]} - \met{X}{[Y,Z]}.
\end{align*}
Since $\ev1,\ev2,\ev3$ fulfil the relations $[\ev i,\, \ev j] = 2\ev i \ev j$ (Clifford multiplication understood), we get the following lemma:
\begin{lemma}\label{nabla}
  We have
  \begin{align*}
    \nabla_{\ev i} \ev i &= 0 \\
    \nabla_{\ev i} \ev j &= \ev i \ev j \quad \text{for $i\neq j$}.
  \end{align*}
\end{lemma}

For $\nabla$ we have a uniquely defined Spin connection $\tilde \nabla$. Using this we can define $\dirac$ on a section $\check \sigma$ of the associated Spin bundle as follows:
\begin{align*}
  \dirac \check \sigma &= c_{\ev1} \tilde \nabla_{\ev1} \check \sigma +  c_{\ev2} \tilde \nabla_{\ev2} \check \sigma +  c_{\ev3} \tilde \nabla_{\ev3} \check \sigma.
\end{align*}

As we have a trivial tangent bundle, we also have a trivial Spin bundle and Clifford bundle. Using a section named $\ev0$, we identify the Spin bundle $\mathcal{S}_{\C} \big(\widetilde P_{\Spe}\big)$ with $\Spe \times \HQ$. Therefore, every section $\check \sigma$ can be written as $\sigma\cdot \ev0$, where $\sigma \in \CSpe$ and the $\cdot$ means $\HQ$-multiplication.

With the help of lemma \ref{nabla} we want to examine the structure of $\tilde \nabla$. Since $\ev0$ is constant in the chosen global trivialisation, we know that
\begin{align*}
  \tilde \nabla_{\ev h} \ev0 &= \frac14 \sum_{i,j} \omega_{ji}(\ev h) c_{\ev i} c_{\ev j}\, \ev0,
\end{align*}
where $\omega_{ji}$ is the matrix of one-forms representing $\nabla$ (see \ref{nabla}). In detail, we get:
\begin{align*}
  \tilde \nabla_{\ev 1} \ev0 &= \frac12 c_{\ev2} c_{\ev3}\, \ev0 = -\,\frac12 \ev1
\end{align*}
and in the same way
\begin{align*}
  \tilde \nabla_{\ev 2} \ev 0 &= -\, \frac12 \ev 2 &  \tilde \nabla_{\ev 3} \ev 0 &= -\, \frac12 \ev 3.
\end{align*}
Hence, for $\dirac \ev0$ we know
\begin{align*}
  \dirac \ev 0 = c_{\ev 1} \left(-\, \frac12 \ev 1\right) + \ldots + c_{\ev 3} \left(-\, \frac12 \ev 3\right) = -\, \frac32 \ev 0.
\end{align*}
If we again represent a section $\check \sigma$ by $\sigma\cdot \ev0$, we get the formula (easy generalisation of the real formula \cite[p.50]{doCarm} to the case $\HQ$):
\begin{align}\label{dspeform}
  \dirac\big(\sigma\cdot \ev0\big) &= \sum_i c_{\ev i} \big(d\sigma(\ev i)\, \ev 0\big) + c_{\ev i} \Big(\sigma\big(\tilde \nabla_{\ev i} \ev 0\big) \Big).
\end{align}
Consider the left action $L_s$ of $\Spe$ on $\CSpe$ given by
\begin{align*}
  L_s \sigma(x) &= \sigma(xs) \quad x,s\in \Spe,
\end{align*}
which induces an infinitesimal action $l_S$ for $S\in \spe$ (this is not a right action although it might look like it was one). To describe the Dirac operator by representation theory we use the following lemma:
\begin{lemma}
  We have
  \begin{align*}
    d\sigma(\ev i) = l_{\ev i} \sigma.
  \end{align*}
\end{lemma}
\begin{proof}
  Let $\tau: (-1/1000,\, 1/1000) \to \Spe$ be a curve with $\tau(0)= e_0$ and $\dot\tau(0) = \ev i$.

From the definition of infinitesimal representations we know that
\begin{align*}
  \big(l_{\ev i} \sigma\big) (x) &= \ddt \Big(\sigma\big(x\,\tau(t)\big)\Big).
\end{align*}
Now $\tau_2(t):= x \tau(t)$ represents a curve with $\tau_2(0) = x$ and $\dot \tau_2(0) = x \cdot \ev i = \ev i$ since $\ev i$ is left-invariant. As $\sigma(x\tau(t)) = \sigma(\tau_2(x))$ and
\begin{align*}
  \ddt \Big(\sigma\big(\tau_2(t)\big)\Big) &= d\sigma_x (\ev i),
\end{align*}
the assertion is proved.
\end{proof}
We plug this into (\ref{dspeform}) and get:
\begin{align*}
  \dirac\big(\sigma\cdot \ev 0\big) &= \sum_i c_{\ev i} \big( (l_{\ev i} \sigma) \ev 0 \big) + \sum_i c_{\ev i} \Big(\sigma\big(\tilde \nabla_{\ev i} \ev 0\big) \Big) \\
&= \sum_i (l_{\ev i} \sigma) \cdot \ev0 \cdot (-\ev i) + \sigma\sum_i c_{\ev i} \big(\tilde \nabla_{\ev i} \ev0\big) \\
&= -\Big(\sum_i (l_{\ev i} \sigma) \cdot \ev i + \frac32\sigma\Big).
\end{align*}

Notice that $c_{\ev i}$ and $\sigma\cdot$ commute since they act on the right and on the left respectively.

From now on, we consider the above trivialisation of $\mathcal{S}_{\C} \big(\widetilde P_{\Spe}\big)$ as implicitly chosen and write
\begin{align}\label{dspedquer}
  \dirac(\sigma) &= \overline \dirac (\sigma) - \frac32 \sigma
  \intertext{with}
  \overline\dirac(\sigma) &= - \sum_i(l_{\ev i} \sigma) \cdot \ev i \quad \text{for $\sigma \in \CSpe$}. \nonumber
\end{align}
In the next section we will investigate the connection between the Dirac operator and the Laplace-Beltrami operator.

\section{$\Delta$ vs. $\dirac$}
\label{sec:s3laplace}

We consider the Laplace-Beltrami operator\index{Laplace-Beltrami operator} $\Delta$ (just called Laplace operator in the following discussion) given by the metric. In our sign convention it will be locally defined as follows:
\begin{defi}[Laplace-Beltrami operator]
  Let $X_j$ be a local basis of vector fields. Then define
  \begin{align*}
    H(\sigma)_{\alpha\beta} &= \nabla_{X_\alpha} \nabla_{X_\beta} \sigma - \nabla_{\nabla_{X_\alpha} X_\beta} \sigma \\[1ex]
    \Delta \sigma &= \sum_{\alpha,\beta} \langle \, , \, \rangle^{\alpha\beta} H(\sigma)_{\alpha\beta},
  \end{align*}
where $\langle\, , \, \rangle^{\alpha\beta}$ represents the $(\alpha,\beta)$-entry of the metric in the given basis.

For the definition of $\Delta$ we view $\sigma$ componentwise (as four functions to $\R$). By standard methods you can show that this definition is independent of the choice of basis.
\end{defi}
\begin{lemma}
  The following two assertions hold:
  \begin{enumerate}[(i)]
  \item $\Delta$ is $\HQ$-linear, i.e. $\Delta(\sigma\cdot \ev i) = \big(\Delta \sigma\big) \cdot \ev i$
  \item $\Delta(l_{\ev i} \sigma) = l_{\ev i} \big(\Delta \sigma\big)$
  \end{enumerate}
\end{lemma}
\begin{proof}
  \begin{enumerate}[(i)]
  \item Clear, since $\Delta$ is linear over the reals and acts componentwise.
  \item We use $\ev1, \ev2, \ev3$ as basis of vector fields. As they are orthonormal, we have $\langle\, , \, \rangle^{\alpha\beta}= \delta_{\alpha\beta}$ (Kronecker symbol).

Therefore, we have
\begin{align*}
  \Delta(l_{e_i}) &= \sum_\alpha \nabla_{\ev \alpha} \nabla_{\ev \alpha} \big(l_{\ev i} \sigma\big)
  \intertext{Since $l_{\ev i} \sigma = d\sigma(\ev i) = \nabla_{\ev i} \sigma$:}
  &=  \sum_\alpha \nabla_{\ev \alpha} \nabla_{\ev \alpha} \nabla_{\ev i} \sigma. \\
\end{align*}
Now we have to exchange $\nabla_{\ev i}$ and $\nabla_{\ev \alpha} \nabla_{\ev \alpha}$ to get $l_{\ev i} \Delta$. This produces some Lie brackets which cancel each other so that the assertion is true.
  \end{enumerate}
\end{proof}
\begin{korr}
  $\dirac \Delta = \Delta \dirac$ on $\CSpe$.
\end{korr}
Hence, $\dirac$ leaves the eigenspaces of $\Delta$ invariant. They are determined in the next section.

\section{The eigenspaces of $\Delta$}
\label{sec:s3eigenD}

The \emph{real} eigenspaces of $\Delta^\R$\index{DeltaR@$\Delta^\R$} can be found in \cite{sakai} (inverting the signs). Theorem 3.13 on page 272 states:
\begin{satz}
  The eigenvalues of $\Delta^\R$ on $(S^3, \langle\, , \, \rangle)$ are given by $\lambda_k = 1-(k+1)^2$, $k\in \Z^+$. The dimension of the eigenspaces is $(k+1)^2$. They are given by the homogeneous harmonic polynomials of degree $k$ on $\R^4$ (which shall be called $H^4_k$\index{H4k@$H^4_k$}).
\end{satz}

To carry this over to $\HQ$ we make the following statement:
\begin{lemma}
  On $\CSpe$, the operator $\Delta$ has only real eigenvalues.
\end{lemma}
\begin{proof}
  From \cite[lemma 3.5, p.266]{sakai} we know:
  \begin{align*}
    \Big\langle \Delta^\R f_1,\, f_2 \Big\rangle_{\text{L}^2} &= \Big\langle f_1,\, \Delta^\R f_2 \Big\rangle_{\text{L}^2},
  \end{align*}
where $\langle\, , \, \rangle_{\text{L}^2}$ denotes the $\text{L}^2$ scalar product.

On $\CSpe$ the $\text{L}^2$ scalar product\index{L2scalar@$\text{L}^2$ scalar product} is defined as
\begin{align*}
  \Big\langle \sigma_1, \, \sigma_2 \Big\rangle_{\text{L}^2} &:= \int_{\Spe} \overline{\sigma_1}\, \sigma_2\, d\nu_{\langle\, \rangle}\quad \in \HQ.
\end{align*}
Since $\Delta$ act componentwise and $\Delta$ is linear we also have 
\begin{align*}
    \Big\langle \Delta \sigma_1,\, \sigma_2 \Big\rangle_{\text{L}^2} &= \Big\langle \sigma_1,\, \Delta \sigma_2 \Big\rangle_{\text{L}^2}.
  \end{align*}
Now if $\lambda_\HQ\in \HQ$ an eigenvalue of $\Delta$ with eigensection $\sigma_\HQ$. A direct calculation shows
\begin{align*}
  \overline{\lambda_\HQ} \int \big| \sigma\big|^2 &= \lambda_\HQ \int \big| \sigma\big|^2.
\end{align*}
Therefore, $\lambda_{\HQ}$ has to be real.
\end{proof}
Hence, we know that every $\lambda_\HQ$ is equal to $\lambda_k$ for one $k\in \Z^+$. The eigenspaces of $\Delta$ are given by \index{Vk@$V_k$}$V_k:= \HQ \otimes_\R H^4_k$. They are right $\HQ$ vector spaces (by multiplication on the first factor).

Now we know that $\dirac|_{V_k}$ maps  $V_k$ to  $V_k$. To understand this operation more precisely, we have to dive into the representation theory of $\Spe \times \Spe$.

\section{The operation of $\Spe \times \Spe$ on $\CSpe$}
\label{sec:s3spespe}

The standard metric of $\Spe$ is left- and right-invariant, so that left and right multiplication act isometrically.

Particularly, the left operation
\begin{align*}
  \beta(s,t)\sigma(x) &= \sigma\big(b(s,t) x \big) \\
  \text{with}\, & b(s,t)x = t^{-1} x s
\end{align*}
is\index{betast@$\beta(s,t)$} a\index{bst@$b(s,t)$} unitary operation (concerning the $\text{L}^2$-structure) of $\Spe\times \Spe$ on $\CSpe$.
\begin{lemma}
  $\beta$ and $\Delta$ commute: $\beta(s,t) \Delta = \Delta \beta(s,t)$.
\end{lemma}
\begin{proof}
  Since $\beta(s,t)$ and $\Delta$ act separately on each component, we only have to look at the real case. Here the assertion is true because $\beta(s,t)$ is an isometry (see \cite[prop 2.4, p.246]{hel})
\end{proof}
So $\beta(s,t)$ acts on each of the eigenspaces $V_k$ of $\Delta$. 

For the rest of the section, $V$ should denote an arbitrary finite dimensional $\beta$-invariant subspace of $\CSpe$. Furthermore, we define
\begin{align*}
  K_b &= \big\{ (s,t) \in \Spe \times \Spe \, \big| \, b(s,t) e_0 = e_0 \big\}.
\end{align*}
The space of \emph{zone functions}\index{zone function} $\zeta_b$ is now defined to be
\begin{align*}
  \zeta_b(V) &= \big\{ \sigma\in V \, \big|\, \beta(s,s)\sigma = \sigma\, \forall(s,s)\in K_b\big\}.
\end{align*}

\begin{satz}\label{zoneexist}
  If $V\neq 0$, then there exists an element $\sigma_\zeta \in \zeta_b(V)$ with $\sigma_\zeta(e_0)\neq0$.
\end{satz}
\begin{proof}
  Define $V^* = \{\sigma^*: V \rightarrow \HQ,\, \text{right-linear} \}$ to be the dual space\index{dual space} of $V$ (as left $\HQ$ vector space). Now define $\xi \in V^*$ to be the map $\xi(\sigma) = \sigma(e_0)$.

Since $V\neq 0$ and $\beta$ acts transitively on $\Spe$, we know that there is an element $\sigma\in V$ with $\sigma(e_0)\neq 0$. Therefore, $\big(\ker \xi\big)$ has (quaternionic) codimension 1.

As $K_b$ fixes the point $e_0$, we know that its induced action leaves the subspace $\big(\ker \xi\big)$ invariant. But then it also has to fix its orthogonal complement $\big(\ker \xi\big)^\perp$. 

So we have an operation of $K_b$ on the one-dimensional space  $\big(\ker \xi\big)^\perp$. Let $\sigma_\zeta$ be an arbitrary basis of this space. Then we have a function $f:\Spe \times \Spe \to \HQ$ with
\begin{align*}
  \beta(s,s) \sigma_\zeta &= \sigma_\zeta \cdot f(s,s) \quad \forall (s,s) \in K_b.
\end{align*}
At $e_0$ we have the following chain of equations:
\begin{align*}
  \sigma_\zeta (e_0) = \big(\beta(s,s) \sigma_\zeta\big)(e_0) = \sigma_\zeta(e_0) \cdot f(s,s).
\end{align*}
We know that $\sigma_\zeta\neq 0$ (since it comes from  $\big(\ker \xi\big)^\perp$) and therefore see
\begin{align*}
  f(s,s)=1 \quad \forall(s,s) \in K_b.
\end{align*}
This shows that the action of $K_b$ on $\sigma_\zeta$ is trivial which proves the theorem.
\end{proof}
We furthermore need the following lemma:
\begin{lemma}
  $K_b$ acts transitively on the unit tangent vectors in $T_{e_0} \Spe$.
\end{lemma}
\begin{proof}
  Following \cite[p.13]{morgan}, we have $\Spe = \Spin(3)$ and the diagonal action of $\Spin(3)$ on $T_{e_0} \Spe$ coincides with the action of $\text{SO}(3)$ on $\R^3$. But this is known to be transitive on the unit sphere.
\end{proof}
\begin{satz}\label{irred}
 The action $\beta$ of $\Spe\times \Spe$ on the eigenspaces $V_k$ of $\Delta$ is irreducible (Idea: \cite{taylor}, p.119).
\end{satz}
\begin{proof}
Assume that $V_k$ splits into the direct sum $V_k^1\oplus V_k^2$ with respect to $\beta$. Following \ref{zoneexist}, there are elements $\sigma^1 \in \zeta\big(V_k^1\big)$ and $\sigma^2 \in \zeta\big(V_k^2\big)$ with $\sigma^i(e_0) \neq 0$. Furthermore, notice that $\sigma^1, \sigma^2\in \zeta\big(V_k\big)$.

\emph{Assertion:} There is a neighbourhood $U$ of $e_0$, so that we have for all $\sigma\in \zeta\big(V_k\big)$:
\begin{align*}
  \dist(e_0,x) = \dist(e_0,y) \quad \Rightarrow \quad \sigma(x) = \sigma(y).
\end{align*}
\begin{proof}[Proof of Assertion]
 Let $\overline{\delta}>0$ be smaller than the radius for which $\exp : \spe \to \Spe$ is bijective and let $U$ be the image of the ball $B(\overline{\delta})$ under this map. $x$ and $y$ shall be any points of distant $\delta < \overline{\delta}$ to $e_0$ in $\Spe$. We denote by $X$ and $Y$ the corresponding points in $\spe$ (under the bijection $\exp$).

Since $K_b$ acts transitively on the unit tangent vectors and preserves distances, it also operates transitively on the surface of $B(\delta)$. If $\overline{b}$ denotes the induced action of $b$ on $\spe$, we find an element $(s,s)\in K_b$, so that $\overline{b}(s,s)X=Y$. This implies $b(s,s)x = y$.

As $\sigma\in \zeta\big(V_k\big)$, we directly see $\sigma(x) = \beta(s,s)\, \sigma(x) = \sigma(y)$.
  \renewcommand{\qedsymbol}{\fbox{Assertion}}
\end{proof}
Choose $U(\eps_0)=\{x\in \Spe\, \big| \, \dist(x,e_0)<\eps_0\}\subset U$ with $\sigma^1(x) \neq 0$ and $\sigma^2(x)\neq0$ for all $x\in U(\eps_0)$. We now look at neighbourhoods $U(\eps)$ with $\eps < \eps_0$. We also define the quaternionic constants $c_\eps^1 = \sigma^1(x_\eps)$ and $c_\eps^2= \sigma^2(x_\eps)$ for an arbitrary element $x_\eps \in \partial U(\eps)$. Now define
\begin{align*}
  \sigma^\eps(x)&:= \sigma^1(x)\big(c_\eps^1\big)^{-1} - \sigma^2(x)\big(c_\eps^2\big)^{-1}\in V_k.
\end{align*}
Being in $V_k$ implies
\begin{align*}
  \big(\Delta - \lambda_k\big)\, \sigma^\eps &\equiv 0 \quad \text{on $U(\eps)$}
\intertext{and we also know:}
  \sigma^\eps &\equiv 0 \quad \text{on $\partial U(\eps)$}.
\end{align*}
\emph{Digression:} 
  In \cite{bersch}, p.152, they consider problems of the following form (Dirichlet problem):
  \begin{align*}
   Lu =  \left(\sum_{\alpha,\beta}^n l_{\alpha\beta}\frac{\partial^2}{\partial x_\alpha \partial x_\beta} + \sum_{\alpha=1}^n l_\alpha \frac{\partial}{\partial x_\alpha} + l\right)u &= f \quad \text{on $\Omega^{\text{o}}$} \\
u &= \phi \quad \text{on $\partial \Omega$}.
  \end{align*}
For that they assume that  $l_{\alpha\beta}$, $l_\alpha$ and $l$ are continuous and that $L$ is conformally elliptic, i.e. there are constants  $c>0$ and $C$ with
\begin{align*}
  \sum_{\alpha, \beta=1}^n l_{\alpha\beta}\xi_\alpha\xi_\beta &\geq c \sum_{\alpha=1}^n \xi_\alpha^2 & \big|l_{\alpha\alpha}\big| &\leq C & \big|l_\alpha\big| &\leq C. 
\end{align*}
I cite (with changed notation) the second theorem on page 153 (\cite{bersch}):
\begin{satz}
  Assume that  $l\leq M$, where $M$ is a positive number and furthermore that the diameter $d$ of $\Omega$ is so small that we have 
  \begin{align*}
    e^{\hat{\alpha}d}-1<\frac1M \quad \text{with $\hat\alpha:= \left(\frac12c\right)\left(C+(C^2+4c)^{\frac12}\right)$}.
  \end{align*}
 Then every solution of the Dirichlet problem fulfils the inequality
  \begin{align*}
    |u| &\leq \frac{\max|\phi| + \left(e^{\hat\alpha d}-1\right) \max |f|}{1- M\left(e^{\hat\alpha d}-1\right)}\, .
  \end{align*}
\end{satz}
~\\[2ex]
Now we want to use this theorem. Let $L$ be the map $\Delta^\R- \lambda_k$ on $U(\eps)$, applied to the components $\sigma_i^\eps$ of $\sigma^\eps$.  For the conformal ellipticity, we choose $c=1$ and $C$ big enough to restrict the coefficient functions on $U(\eps_0)$ (and thus also on every $U(\eps)$ with $\eps<\eps_0$). The functions $f$ and $\phi$ are chosen to be zero.

Let furthermore $M$ be defined as $-\lambda_k$ and $\eps$ be so small that the diameter condition is fulfilled. Then the inequality implies $\sigma^\eps_i=0$ for all components. Therefore, we have on $U(\eps)$:
\begin{align*}
  \sigma^1(x)\big(c_\eps^1\big)^{-1} - \sigma^2(x)\big(c_\eps^2\big)^{-1} =0.
\end{align*}
Since $\sigma^1$ and $\sigma^2$ are polynomials and therefore analytic, this equation is true on the whole of $\Spe$. But then $\sigma^1$ and $\sigma^2$ are linearly independent over $\HQ$, which contradicts our assumption.
\end{proof}

Our next aim is to classify the $\HQ$-representations of $\Spe \times \Spe$. We will use this to determine which of the ``abstract'' representations of our list corresponds to the representations on $V_k$ just found.

\section{Representations of $\Spe \times \Spe$}
\label{sec:s3rep}

The groups $\Spe$ and $\text{SU}(2)$ are isomorphic, so we can use \cite[p.76]{brocktom} to give a classification of the \emph{complex} (left) representations of $\Spe$:

For that, let $\HQ_1$ be the canonical inverse right representation of $\Spe$ on $\HQ$ (coming from the quaternionic multiplication). As before we consider $\HQ$ as $\C^2$ using the basis $e_0,e_2$ and note that the (left) action of $\ci$ commutes with the representation of $\Spe$.

Now let $\HQ_k$\index{Hk@$\HQ_k$} be the $k$-fold complex symmetric tensor product $\HQ_1 \circledcirc \ldots \circledcirc \HQ_1$ with the induced representation of $\Spe$. We choose the basis
\begin{align*}
  \ket d = e_0 \circledcirc \ldots \circledcirc e_0 \circledcirc e_2 \circledcirc \ldots \circledcirc e_2 \quad d=0,\ldots, k,
\end{align*}
where $d$ is the number of $e_2$-terms. 
\index{ketd@$\ket{d}$}The spaces $\HQ_k$ with the standard metric from $\C^2$ form exactly the irreducible complex representations of $\Spe$ (following \cite{brocktom}).

Hence, the irreducible complex representations of $\Spe \times \Spe$ are given by $\HQ_i \otimes_\C \HQ_j$, $i,j\in \Z^+$ (see e.g. \cite[theorem 3.65, p.70]{adams}).

From theorem 3.57, p.66 of \cite{adams} we know

\begin{satz}
  For a compact Lie group we can find families of real representations $\big[\R\big]_m$, complex representations $\big[\C\big]_n$ and quaternionic representations $\big[\HQ\big]_p$, so that we have:

The non-equivalent irreducible representations are exactly
\begin{enumerate}[(i)]
\item $\big[\R\big]_m$, $\big[r^\downarrow \C\big]_n$, $\big[r^\downarrow c^\downarrow \HQ\big]_p$ over $\R$.
\item $\big[c^\uparrow \R\big]_m$, $\big[\C\big]_n$, $\overline{\big[\C\big]}_n$, $\big[c^\downarrow \HQ\big]_p$ over $\C$.
\item $\big[h^\uparrow c^\uparrow \R\big]_m$, $\big[h^\uparrow \C\big]_n$, $\big[\HQ\big]_p$ over $\HQ$.
\end{enumerate}
The term $\overline{\big[\C\big]}_n$ means ${\big[\C\big]}_n$ with conjugated scalar multiplication, $c^\downarrow$ and $r^\downarrow$ view the respective spaces as complex or real space (i.e. forget some structure), whereas $c^\uparrow$ and $h^\uparrow$ tensor by $\C \otimes_\R$ or $\HQ \otimes_\C$ respectively.
\end{satz}

Adams also outlined how to find the three classes of representations: Decompose the irreducible complex representations into those which are not self-conjugate (they form $\big[\C\big]_n$) and those which are self-conjugate, to be divided further in real and quaternionic ones.

At first we will analyse this for $\Spe$: Since in every complex dimension there is  exactly one representation, we know that all representations have to be self-conjugate. We distinguish two cases:
\begin{center}\begin{tabular}{rp{11cm}}
\emph{$k$ is even} & Here $\HQ_k$ has odd complex dimension, so it cannot be quaternionic and must be real. \\
\emph{$k$ is odd} & To show that $\HQ_k$ is quaternionic, we have to find a complex antilinear map $J: \HQ_k \to \HQ_k$ with $J^2 = -\id$. We define
{\begin{align*}
  J\big(\ket d\big) &= (-1)^{d} \ket{k-d},
\end{align*}}
which obviously fulfils the condition when we continue it in a complex anti-linear fashion.
\end{tabular}\end{center}
Hence, we know that
\begin{align*}
  \big[c^\uparrow \R\big]_m &= \HQ_{2m} & \big[c^\downarrow \HQ\big]_p &= \HQ_{2p+1}.
\end{align*}
After considering $\Spe$, we move to $\Spe \times \Spe$. To distinguish the representation spaces from the ones above, we call them $\big[\R^\times\big]_m$, $\big[\C^\times\big]_n$, $\big[\HQ^\times\big]_p$; they appear as products of the representations of $\Spe$.

The product of two self-conjugate spaces is always self-conjugated: This rules out spaces of the form $\big[\C^\times\big]_n$. Furthermore the real representations are the product of two representations of the same kind (i.e. same division algebra), while the quaternionic ones are created by two representations of different kind.

All in all the representations of $\Spe\times \Spe$ over $\HQ$ are:
\begin{align*}
  \HQ &\otimes_\C \Big(\big[c^\uparrow \R\big]_{m_1} \otimes_\C \big[c^\uparrow \R\big]_{m_2}\Big) & \big[\R\big]_{m_1} &\otimes_\R \big[\HQ\big]_{p_2} \\[1ex]
 \HQ &\otimes_\C \Big(\big[c^\downarrow \HQ \big]_{p_1} \otimes_\C \big[c^\downarrow \HQ\big]_{p_2}\Big) & \big[\HQ\big]_{p_1} &\otimes_\R \big[\R\big]_{m_2}.
\end{align*}

Now we want to detect which of these representations is equivalent to the one given by $V_k$.

We look at the transposition map
\begin{align*}
  \Spe\times\Spe &\stackrel{\circlearrowright}{\longrightarrow} \Spe\times\Spe \\
(s,t) &\mapsto (t,s).
\end{align*}

Furthermore, let $\text{inv}: \Spe \to \Spe$ be the isometry given by $x \mapsto x^{-1}$. This induces an isometry $\text{inv}^\infty: \CSpe \to \CSpe$. Since isometries commute with $\Delta$ (see above) we get a map
\begin{align*}
  \text{inv}^\infty|_{V_k} : V_k \to V_k.
\end{align*}
$\text{inv}^\infty|_{V_k}$ is therefore a self-inverse isomorphism of $V_k$. Look at the diagram:

\begin{center}
~\xymatrix @C=2.5cm{
\big(\Spe\times\Spe\big) \times V_k
\ar[rr]^-{\beta}
\ar[d]_{\id\times \text{inv}^\infty}
&& V_k \ar[d]^{\text{inv}^\infty}\\
(\Spe\times\Spe) \times V_k
\ar[rr]^-{\beta\circ\big(\circlearrowright \times \id\big)}
&& V_k
} 
\end{center}
This diagram commutes since
\begin{align*}
  \text{inv}^\infty \big(\beta(s,t)\sigma\big) (x) &= \big(\beta(s,t)\sigma\big) \big(x^{-1}\big) = \sigma\big(t^{-1}x^{-1} s\big)
\intertext{and}
\beta \circ \big(\circlearrowright \times \id\big) \circ \big(\id \times \text{inv}^\infty\big) \big(s,t,\sigma\big)(x) &= \beta(t,s)\big(\text{inv}^\infty(\sigma)\big) (x) = \text{inv}^\infty \big(\sigma\big) \big(s^{-1} x t \big) = \sigma\big(t^{-1} x^{-1} s \big).
\end{align*}
Hence, the representations $\beta$ and $\beta \circ \big(\circlearrowright \times \id\big)$ are equivalent. The representation $\gamma$ from the list which corresponds to $\beta$ must have the same property: Since $\circlearrowright$ exchanges the arguments of $\Spe\times \Spe$, $\gamma$\index{gamma@$\gamma(s,t)$} can only be one of the symmetric representations
\begin{align*}
  \HQ &\otimes_\C \Big( \big[c^\uparrow \R\big]_m \otimes_\C  \big[c^\uparrow \R\big]_m\Big) \\
  \HQ &\otimes_\C \Big( \big[c^\downarrow \HQ\big]_p \otimes_\C \big[c^\downarrow \HQ\big]_p\Big).
\end{align*}
This family can also be expressed as
\begin{align*}
  \HQ \otimes_\C \big(\HQ_k \otimes_\C \HQ_k\big) \quad \forall k\geq1.
\end{align*}
These spaces have the respective $\HQ$-dimension $(k+1)^2$. So we must have
\begin{align*}
  V_k &\cong  \HQ \otimes_\C \big(\HQ_k \otimes_\C \HQ_k\big) \quad \forall k\geq1,
\end{align*}
where by $\cong$ we mean an isometry of right $\HQ$ vector spaces which commutes with the respective representations $\beta$ and $\gamma$. This abstract isomorphism will be calculated concretely in \ref{sec:s3d}. Until then we use the abstract isomorphism to calculate the eigenvalues of $\dirac$.

\section{The spectrum of $\dirac$}
\label{sec:s3dspec}

As $\overline\dirac(\sigma) = - \sum_i(l_{\ev i} \sigma) \cdot \ev i$ acts on $V_k$ only be means of scalar multiplication and representation of $\spe$, we can replace $V_k$ by $\HQ \otimes_\C \big(\HQ_k \otimes_\C \HQ_k\big)$ for the calculation of the eigenvalues (Notice that $L$ can be written as $L_s = \beta(s,e_0)$).

As representation space for $\Spe$, the space $\HQ \otimes_\C \big(\HQ_k \otimes_\C \HQ_k\big)$ decomposes into the direct sum of the irreducible representations
\begin{align*}
  \HQ^q_k &:= \HQ \otimes_\C \Big(\HQ_k \otimes_\C \spann\big\{\ket q\big\}\Big).
\end{align*}
Here we want to explicitly calculate $\overline\dirac$. For that purpose we choose a complex basis $e_r \otimes \ket p$, $r=0,2$. 

The space $\HQ^q_k\cong \HQ \otimes_\C \HQ_k$ is not only a complex vector space, but is a right $\HQ$ vector space by multiplication on the first component and a quaternionic representation space for $\Spe$ by inverse right multiplication on the second component. The representation $\gamma(s, e_0)$, which corresponds to $\beta(s,e_0)=L_s$, will still be denoted by $L$.

We now want to examine the action of $l_{\ev i}$ on $e_r \otimes \ket p$. We will leave out the $e_r$-part since it is unimportant for the representation:
\begin{align*}
  l_{\ev1} \big(\ket p \big) &= \left(\ddt \exp(t\ev 1) e_0 \circledcirc \ldots \circledcirc \exp(t\ev 1) e_2 \right).
\end{align*}
Since $\Spe$ is a matrix group, we can calculate its exponential in the usual fashion:
\begin{align*}
  \exp(t\ev i) &= 1+ te_i + \frac{(te_i)^2}{2} + \ldots
\end{align*}
Algebraically $\ddt$ means that we only look at the linear part in $t$:
\begin{align}\label{le1}
  l_{\ev1} \big(\ket p \big) &= -(k-p)\, e_1 \circledcirc e_0 \ldots \circledcirc e_2 + p \, e_0 \circledcirc \ldots \circledcirc e_3 \circledcirc e_2 \circledcirc \ldots \circledcirc e_2\nonumber \\
  &= -(k-p)\ci \, \ket{p} + p\ci \ket{p}.
\end{align}

In the same way, we get:
\begin{align}\label{le2}
  l_{\ev 2} \ket p &= (p-k)\, \ket{p+1} + p\, \ket{p-1} \\
\label{le3}  l_{\ev 3} \ket p &= (p-1)\ci \, \ket{p+1} - p\ci\, \ket{p-1}.
\end{align}
Notice: In these and the following formulas $\ket{-1}$ and $\ket{k+1}$ are assumed to be zero.
Now we calculate $\overline\dirac^2$:
\begin{align*}
  \overline\dirac^2 &= \sum_{i,j} l_{\ev j} l_{\ev i} \sigma \cdot \ev i \cdot \ev j \\
  &= - l_{\ev 1}^2 \sigma - l_{\ev 2}^2 \sigma- l_{\ev 3}^2 \sigma \\
  &\quad +\big(l_{\ev3}l_{\ev2} - l_{\ev 2} l_{\ev 3}\big) \sigma \cdot \ev1 + \big(l_{\ev1}l_{\ev3} - l_{\ev 3} l_{\ev 1}\big) \sigma \cdot \ev2 + \big(l_{\ev2}l_{\ev1} - l_{\ev 1} l_{\ev 2}\big) \sigma \cdot \ev3.
\end{align*}
We have
\begin{align*}
  l_{\ev i} l_{\ev j} - l_{\ev j} l_{\ev i} = l_{[\ev i, \ev j]}= l_{2\ev i \ev j}.
\end{align*}
If we plug this in, we get:
\begin{align*}
  \overline\dirac^2 &=  - l_{\ev 1}^2 \sigma - l_{\ev 2}^2 \sigma- l_{\ev 3}^2 \sigma \\
  &\quad -2l_{\ev1} \sigma \ev 1 - -2l_{\ev2} \sigma \ev 2 - -2l_{\ev3} \sigma \ev 3,
\end{align*}
which means, that
\begin{align*}
  \big(\overline\dirac^2 - 2\overline\dirac\big)\sigma = -l_{\ev1}^2\sigma -l_{\ev2}^2\sigma -l_{\ev3}^2\sigma.
\end{align*}
\begin{lemma}
  On $\HQ^q_k$, we have
  \begin{align*}
    \big(-l_{\ev1}^2  -l_{\ev2}^2-l_{\ev3}^2\big) = k(k+2)\, \id.
  \end{align*}
\end{lemma}
\begin{proof}
  We take the basis $e_r \otimes \ket{p}$ and use the formulas (\ref{le1}), (\ref{le2}) and (\ref{le3}) twice.
\end{proof}
As a result we have the quadratic equation
\begin{align}\label{dquad}
  \big(\overline \dirac + k\big) \big(\overline\dirac - (k+2)\big) \sigma &= 0 \quad \forall \sigma\in \HQ^q_k.
\end{align}
Therefore, every element $\sigma\in \HQ^q_k$ generates a one- or two-dimensional complex $\overline\dirac$-invariant subspace.

On our basis we have the following two formulas for $\overline\dirac$:
\begin{align*}
  \overline\dirac\big(e_0\otimes\ket p \big) &= (2p-k)\, e_0 \otimes \ket p - 2p e_2 \otimes \ket{p-1} \\
  \overline\dirac\big(e_2\otimes\ket p \big) &= -(2p-k)\, e_2 \otimes \ket p - 2(k-p) e_0 \otimes \ket{p-1}.
\end{align*}
Hence we get the invariant subspaces
\begin{align*}
  \Big\{ e_0 \otimes \ket{p}, \, e_2 \otimes \ket{p-1}\Big\}
\end{align*}
for $p=0,\ldots, k+1$, where the first and last one are one-dimensional. (\ref{dquad}) gives us two eigenvectors in every of the two-dimensional subspaces, one for $-k$ and one for $k+2$.

If we subtract $\frac32$ to get from $\overline\dirac$ to $\dirac$ and go back to the whole space $\HQ \otimes_\C \big(\HQ_k \otimes_\C \HQ_k\big)$, we get the following complex orthogonal basis (leaving out the $\otimes$ in the usual ket-manner):
\begin{align*}
  \intertext{For $k+\frac12$:}
  e_0 \otimes \ket p\, \ket q - e_2 \otimes \ket {p-1}\, \ket q &\quad p=1,\ldots,k \quad q=0, \ldots, k
\intertext{For $-k-\frac32$:}
(p-k-1) e_0 \otimes \ket p \, \ket q - p e_2 \otimes \ket{p-1} \, \ket q &\quad p=1,\ldots, k \quad q=0,\ldots k \\[0.5ex]
e_0 \otimes \ket0 \, \ket q, \, e_2 \otimes \ket k \, \ket q &\quad q=0,\ldots, k
\end{align*}

Now we want to translate these abstract basis vectors into concrete ones; for that purpose we need a computable isomorphism from $V_k$ to $\HQ \otimes_\C \big(\HQ_k \otimes_\C \HQ_k\big)$.

\section{An eigenbasis for $\dirac$}
\label{sec:s3d}
Since  $V_k$ and $\HQ \otimes_\C \big(\HQ_k \otimes_\C \HQ_k\big)$ are both quaternionic representations which come from real representations, we can compare them equally well on the real or complex level; for simplicity we choose the latter one: The aim of this section will be to find an isomorphism between the $\C$-representations $\beta$ on $W_k:= H^4_k \otimes_\R \C$\index{Wk@$W_k$} and $\gamma$ on $\HQ_k\otimes_\C \HQ_k$.

Both of them induce representations of the Lie algebra $\spe\oplus \spe$. Since we look at complex representations, we get a canonical representation of the complexified Lie algebra $\C \otimes \big(\spe \oplus \spe\big)= \slc \oplus \slc$\index{slc@$\slc$} (compare \cite[p.102]{hall}); those complexified representations shall be called $\overline\beta$ and $\overline\gamma$\index{betao@$\overline\beta$}\index{gammao@$\overline\gamma$}.

As a basis for $\slc$, we choose
\begin{align*}
   \Hsl &:= \matrixtt{1}{0}{0}{-1} &\Xsl &:= \matrixtt0100 & \Ysl &:= \matrixtt0010 
\end{align*}

As described in the proof of theorem D.1 on page 322 \cite{hall}, there exists a basis
\begin{align*}
  v_{-k}, v_{-k+2}, \ldots, v_{k-2}, v_k \intertext{for $\HQ_k$ so that}
  \overline\gamma\big(\Hsl, \Hsl\big) (v_a\otimes v_b) &= (a+b)\, v_a \otimes v_b.
\end{align*}
In this decomposition of $\HQ_k$ into eigenspaces of $\overline\gamma\big(\Hsl, \Hsl\big)$, we see that the one for the eigenvalue $2k$ is 1-dimensional.

We now want to determine the eigenspace for $2k$ of $\overline\gamma\big(\Hsl, \Hsl\big)$ and $\overline\beta\big(\Hsl, \Hsl\big)$ and use this for defining an isomorphism.

Notice that $\Hsl$ corresponds to $\ci \otimes e_1$, using the isomorphism $\slc \cong \C \otimes \spe$. We have
\begin{align*}
  \overline\gamma\big(\Hsl, \Hsl\big)\, \ket0 \, \ket0 &= \ci \big(e_1\circledcirc e_0 \circledcirc \ldots \circledcirc e_0 \ldots  \big) \\
  &= 2k\, \ket0 \, \ket0.
\end{align*}
This shows that $\ket0\, \ket0$ is a normed basis vector of the $2k$-eigenspace.

Now we look at $W_k$. It consists of complex harmonic homogeneous polynomials of degree $k$ in $x_0, x_1, x_2, x_3$. We write $z=(z_1,z_2) = (x_0,x_1,x_2,x_3)$ with $z_1=x_0+  x_1\ci$ and $z_2 = x_2 + x_3\ci$.

Our strategy is to solve the problem combinatorically by using the functions
\begin{align*}
  g_2(z) &= z_2 & \overline{g_2}(z) &= \overline z_2 \\
  g_{-1}(z) &= -z_1 & \overline{g_1} (z) &= \overline z_1.
\end{align*}
\index{gi@$g_2$ etc.} We compute
\begin{align*}
  \overline\beta(\ev1,\ev1) g_2(z) &= \ddt g_2 \big(\exp(-t\ev1) z \exp(t\ev1)\big) \\
  &= \ddt g_2 \big(z+ t(-\ev1z+z\ev1) + t^2\cdot \ldots \big) \\
  &= g_2\Bigg( {-\ev1 z_1 + z_1 \ev1 \choose -\ev1z_2+ z_2(-\ev1)}\Bigg) \\
  &= -2\ev1\, z_2.
\end{align*}
Since $\ci \cdot \overline \beta(\ev1,\ev1) = \overline\beta(\Hsl, \Hsl)$, we know that
\begin{align*}
  \overline\beta\big(\Hsl, \Hsl\big)\, g_2(z) &= 2g_2(z).
\end{align*}
Furthermore, you can see explicitly that $g_2^k(z)$ is a harmonic polynomial of degree $k$, which means $g_2^k\in W_k$. Following the usual formula for products we have
\begin{align*}
  \overline\beta\big(\Hsl, \Hsl\big)\, g_2^k &= 2k \cdot g_2^k.
\end{align*}
Therefore, we know that $g_2^k$ generates the $2k$-eigenspace of $\overline\beta\big(\Hsl, \Hsl\big)$. To calculate the norm of $g_2^k$ we use the following integral formula:
\begin{lemma}
  For $l_1,l_2,l_3,l_4\in \Z_{\geq 0}$, we have
  \begin{align*}
    \int_{S^3} g_2^{l_1} \overline{g_2}^{l_2} g_{-1}^{l_3} \overline{g_1}^{l_4}\, d\nu_{\langle\, \rangle} &= \begin{cases}
      2\pi^2(-1)^{l_4}\, \frac{l_1!l_3!}{(l_1+l_3+1)!} & \text{for $l_1=l_2$ and $l_3=l_4$} \\
      0 &\text{otherwise}
      \end{cases}
  \end{align*}
\end{lemma}
\begin{proof}
%Cut out two circles (with measure zero) and identify the rest of $S^3$ with $(S^1)^2\times ]0,1[$. \fehlt (short version of the calculation).
 Let $S^3= \big\{(z_1,z_2)\, \big|\, |z_1|^2+|z_2|^2=1 \big\}$ and let
  \begin{align*}
    N:= \big\{ (z_1,z_2)\in S^3 \, \big| \, z_1=0 \vee  z_2=0 \big\}.
  \end{align*}
As $N$ has measure zero, we have
\begin{align*}
  \int_{S^3} f\,*1 = \int_{S^3 \setminus N} f\,*1
\end{align*}
for all functions $f:S^3 \rightarrow \C$ which are integrable.

We consider the map

\begin{eqnarray*}
  \rrspace  &\stackrel{\eta}{\longrightarrow} & S^3 \setminus N \\
(t\quad,\quad s \quad, \quad \rho) & \longmapsto & \big(e^{t\ci} \cdot \sqrt{\rho}\,,\, e^{s\ci} \cdot \sqrt{1-\rho} \big) \\
\left(-\ci\log^{-1}\left(\frac{z_1}{|z_1|}\right), -\ci\log^{-1}\left(\frac{z_2}{|z_2|}\right), |z_1|^2\right) & \longmapsfrom & (z_1, z_2)
\end{eqnarray*}
This is a diffeomorphism with given inverse. Thus we have:
\begin{align*}
  \int_{\rrspace} (f\circ \eta)\cdot \eta^*(*1_{S^3}) = \int_{S^3\setminus N} f\,*1.
\end{align*}
A direct calculation shows
\begin{align*}
  \eta^*(*1_{S^3}) = \eta^*(\ev1^* \wedge \ev2^* \wedge \ev3^*) = \frac12 dt \wedge ds \wedge d\rho.
\end{align*}
Using this, we can reduce everything to a standard integral which is computable by the gamma function (see \cite[p.294]{friedman}).
\end{proof}
Using this formula we have
\begin{align*}
  \big\langle g_2^k, \, g_2^k\big\rangle_{\text{L}^2} &= \int_{S^3} g_2^k \overline{g_2}^k\, d\nu_{\langle\, \rangle} = 2\pi^2 \cdot \frac{k!}{(k+1)!},
\end{align*}
which means $\|g_2^k\|_{\text{L}^2} = \sqrt{\frac{2\pi^2}{k+1}}$.

As $\overline\beta$ and $\overline\gamma$ are isomorphic as representations, our isomorphism has to map
\begin{align*}
  \ket0 \,\ket0 \quad \text{onto} \quad u\sqrt{\frac{k+1}{2\pi^2}}\, g_2^k,
\end{align*}
where $u\in S^1$. But since multiplication with $u$ is an isomorphism of complex representations, we can choose $u$ to be anything we like (and take $u=1$).
\begin{defi}
  Let $\mathcal{I}$ be the isomorphism of complex $\Spe\times \Spe$ representations which maps $\ket0 \, \ket0$ onto $u\sqrt{\frac{k+1}{2\pi^2}}\, g_2^k$.
\end{defi}
The map $\mathcal{I}$ is uniquely defined because it commutes with the action of $\Spe \times \Spe$; this action creates a generating system out of every non-zero vector due to irreducibility.

Now want to compute $\mathcal{I}$ on the basis $\ket p \, \ket q$. For that purpose we use the action of $\big(\Ysl, 0\big)$ and $\big(0,\Ysl\big)$. In $\C\otimes \spe$ we have
\begin{align*}
  \Ysl &= \frac12\big((-1)\otimes \ev2 + \ci \otimes \ev3\big).
\end{align*}
Therefore, we get
\begin{align*}
  \overline\gamma\big(\Ysl,0\big)\, \ket p \, \ket q &= -\, \frac12 \overline\gamma(\ev2,0)\ket p \, \ket q - \frac\ci2\, \overline\gamma(\ev3,0) \ket p \, \ket q \\
  &= -\, \frac12 (p-k) \ket{p+1}\, \ket q -\, \frac12 p \ket{p-1} \, \ket q \\
  & \quad -\, \frac12(p-k)\ket{p+1}\, \ket q + \frac12 p\ket{p-1}\, \ket q \\
  &= (k-p)\, \ket{p+1}\, \ket{q}.
\end{align*}
In the same matter $\overline\gamma(0,\Ysl)\, \ket p \, \ket q = (k-q)\, \ket p \, \ket{q+1}$.

This gives us a method to calculate $\ket p \, \ket q$ out of $\ket 0 \, \ket 0$ in $p+q$ steps.

To compute $\mathcal{I}\big(\ket p \, \ket q\big)$ we need to understand $\overline\beta\big(\Ysl,0\big)$ and $\overline\beta\big(0,\Ysl\big)$:
\begin{lemma}
  If we describe $\overline\beta\big(\Ysl,0\big)$ by a left-down arrow and $\overline\beta\big(0,\Ysl\big)$ by a right-down arrow, we get the following diagram
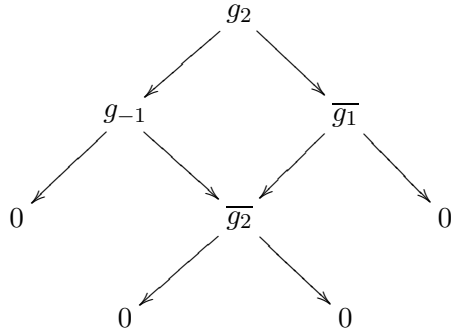
\begin{align*}
  \xymatrix{ 
&& g_2 \ar[ld] \ar[rd]&& \\
&g_{-1}\ar[ld]\ar[rd]&&\overline{g_1}\ar[ld]\ar[rd] & \\
0&& \overline{g_2}\ar[ld]\ar[rd] && 0 \\
& 0 && 0 &}
\end{align*}
\end{lemma}
\begin{proof}
  This is a direct calculation.
\end{proof}
With the help of the lemma above we can show
\begin{satz}\label{ds3basis}
  Let $\mathcal{I}$ be the map defined above. Then we have for $p,q\in \{0,\ldots, k\}$:
  \begin{align*}
    \begin{array}{rcl}
      \HQ_k \otimes \HQ_k & \stackrel{\mathcal{I}}{\longrightarrow} & W_k \\[1ex]
      \ket p\, \ket q & \mapsto & {k \choose p}^{-1} {k \choose q}^{-1}\sqrt{\frac{k+1}{2\pi^2}}\, \sum_{i=0}^k {k \choose k-q-i, p-i, i}\,g_2^{k-q-i} \overline{g_2}^{p-i} g_{-1}^i \overline{g_1}^{q-p+i}
    \end{array}
  \end{align*}
\end{satz}
\begin{proof}
  By induction. %(CHECK exponents \fehlt)
\end{proof}
With the help of the theorem, we can translate the basis in $\HQ \otimes_\C \big(\HQ_k \otimes_\C \HQ_k\big)$ into a basis in $V_k$.

The terms in $W_k$ can be further examined by the rich theory of harmonic polynomials (see e.g.  \cite{axler}).

\bibliographystyle{elsarticle-num}
\bibliography{hitchin}

\end{document}